\documentclass[11pt]{amsart}

\usepackage{latexsym}
\usepackage{amssymb}
\usepackage{amsmath,mathabx}

\newtheorem{theorem}{Theorem}[section]
\newtheorem{lemma}[theorem]{Lemma}
\newtheorem{proposition}[theorem]{Proposition}

\theoremstyle{definition}
\newtheorem{definition}[theorem]{Definition}

\theoremstyle{remark}
\newtheorem{remark}[theorem]{Remark}
\newtheorem{example}[theorem]{Example}

\numberwithin{equation}{section}

\newcommand{\C}{\mathbb{C}}
\newcommand{\R}{\mathbb{R}}
\newcommand{\N}{\mathbb{N}}
\newcommand{\T}{\mathbb{T}}
\newcommand{\E}{\mathcal{E}}
\newcommand{\lm}{\lambda}
\newcommand{\f}{\varphi}
\newcommand{\Ad}{\operatorname{Ad}}
\newcommand{\WAP}{\operatorname{WAP}}
\newcommand{\wap}{\operatorname{wap}}
\newcommand{\AP}{\operatorname{AP}}
\newcommand{\ap}{\operatorname{ap}}
\newcommand{\cbr}{C_{b,r}}
\newcommand{\ctr}{\operatorname{Ctr}}
\newcommand{\Int}{\operatorname{Int}}
\newcommand{\ep}{\varepsilon}

\newcommand{\B}{\mathcal B}
\newcommand{\Pp}{\mathcal P}
\newcommand{\Om}{\Omega}
\newcommand{\Hh}{\mathcal H}
\newcommand{\K}{\mathcal K}
\newcommand{\Ll}{\mathcal L}
\newcommand{\om}{\omega}
\newcommand{\mf}{\mathfrak{m}}
\newcommand{\co}{\operatorname{co}}

\begin{document}

\title[(Weakly) almost periodic functions and von Neumann algebras]{Almost and weakly almost periodic functions on the unitary groups of von Neumann algebras}

\author{Paul Jolissaint}
\address{Universit\'e de Neuch\^atel,
       Institut de Math\'ematiques,       
       E.-Argand 11,
       2000 Neuch\^atel, Switzerland}
       
\email{paul.jolissaint@unine.ch}


\date{\today}


\begin{abstract}
Let $M\subset B(\Hh)$ be a von Neumann algebra acting on the (separable) Hilbert space $\Hh$. We first prove that $M$ is finite if and only if, for every $x\in M$ and for all vectors $\xi,\eta\in\Hh$, the coefficient function $u\mapsto \langle uxu^*\xi|\eta\rangle$ is weakly almost periodic on the topological group $U_M$ of unitaries in $M$ (equipped with the weak operator topology). The main device is the unique invariant mean on the $C^*$-algebra $\WAP(U_M)$ of weakly almost periodic functions on $U_M$. Next, we prove that every coefficient function $u\mapsto \langle uxu^*\xi|\eta\rangle$ is almost periodic if and only if $M$ is a direct sum of a diffuse, abelian von Neumann algebra and finite-dimensional factors. Incidentally, we prove that if $M$ is a diffuse von Neumann algebra, then its unitary group is minimally almost periodic.
\end{abstract} 

\maketitle

\begin{flushright}
\textit{Dedicated to Pierre de la Harpe}
\end{flushright}

\section{Introduction}

The present work is inspired by the article \cite{HarpeMoy} where P. de la Harpe proved that, if $M$ is a von Neumann algebra with separable predual, then it is Approximately Finite-Dimensional (AFD) if and only if there exists a left invariant mean on the $C^*$-algebra $\cbr(U_M)$ of right uniformly continuous functions on the unitary group $U_M$ of $M$; in other words, $M$ is AFD if and only if the Polish group $U_M$ is amenable. In fact, he used the existence of a left invariant mean on $\cbr(U_M)$ to show that $M$ has Schwartz's property P \cite[Definition 1]{Schwartz}, the latter being equivalent to injectivity hence to approximate finite-dimensionality by \cite{Connes}.

For every topological group $G$, there is a space of continuous functions (in fact a $C^*$-algebra) on which a unique bi-invariant mean always exists: it is the set of all weakly almost periodic functions on $G$. The aim of the present article is then to exploit the existence of such a mean on the group $U_M$. It turns out that it provides a characterization of finite von Neumann algebras. 

In order to present the content of this article, let us recall some definitions and fix notation.

First, let $G$ be a topological group. We denote by $C_b(G)$ the $C^*$-algebra of all bounded, continuous, complex-valued functions on $G$ equipped with the uniform norm $\Vert f\Vert_\infty:=\sup_{s\in G}|f(s)|$. 
For $g\in G$ and $f:G\rightarrow\C$, we denote by $g\cdot f:G\rightarrow\C$ (resp. $f\cdot g$) the left (resp. right) translate of $f$ by $g$, i.e.
\[
(g\cdot f)(s)=f(g^{-1}s)\quad\textrm{and}\quad
(f\cdot g)(s)=f(sg)
\]
for all $f:G\rightarrow\C$ and $g,s\in G$. The corresponding left (resp. right) orbit is denoted by $Gf$ (resp. $fG$). A function $f\in C_b(G)$ is \textit{right uniformly continuous} if $\Vert g\cdot f-f\Vert_\infty\to 0$ as $g\to 1$. The subset of all right uniformly continuous functions $f\in C_b(G)$ is a $C^*$-subalgebra of $C_b(G)$ denoted by $\cbr(G)$, and it contains all right translates of all its elements. Our definition follows that of P. de la Harpe \cite{HarpeMoy} and F.P. Greenleaf \cite{Greenleaf}, but not that of P. Eymard \cite{EymardMoy} for instance.

A function $f\in \cbr(G)$ is \textit{weakly almost periodic} if its orbit $Gf$ is weakly relatively compact in $\cbr(G)$. An application of Hahn-Banach Theorem shows that the orbit is weakly relatively compact in $\cbr(G)$ if and only if it is weakly relatively compact in the larger Banach space $C_b(G)$.

It follows from \cite[Proposition 7]{Grothen} that $Gf$ is weakly relatively compact if and only if $fG$ is (see also Proposition \ref{criterewap} in the appendix). The set of all weakly almost periodic functions on $G$ is denoted by $\WAP(G)$; it is a $C^*$-subalgebra of $\cbr(G)$, and its main feature, which will play a central role here, is the existence of a unique left and right $G$-invariant mean $\mf$ on $\WAP(G)$. We are indebted to S. Knudby for having indicated to us that F.P. Greenleaf's monograph \cite{Greenleaf} contains a proof of that result for locally compact groups, which relies on Ryll-Nardzewski Theorem, but we think that it is worth presenting (for the reader's convenience) a self-contained proof for arbitrary topological groups in the appendix of the present article. Our proof also uses Ryll-Nardzewski Theorem.

Next, let $M\subset B(\Hh)$ be a von Neumann algebra acting on the separable Hilbert space $\Hh$. Our references for von Neumann algebras are the monographs \cite{DivN} and \cite{Tak1}. 
We denote by $U_M$ the group of unitary elements of $M$. It is a topological group, and even a Polish group, when endowed with any of the following equivalent topologies on it: the weak, ultraweak, strong and ultrastrong operator topologies.

For $T\in B(\Hh)$ and $\xi,\eta\in\Hh$, we define the associated \textit{coefficient function} $\xi\star T\star \eta$ on $U_M$ as follows:
\[
\xi\star T\star \eta(u):=\langle uTu^*\xi|\eta\rangle
\]
for every $u\in U_M$.

Here is the first of our main results; its proof is contained in \S 3.

\begin{theorem}\label{1.1}
Let $M\subset B(\Hh)$ be a von Neumann algebra. Then it is finite if and only if, for every $x\in M$, all its coefficient functions $\xi\star x\star \eta$
are weakly almost periodic. If it is the case, there exists a conditional expectation $\E_M:C^*(M,M')\rightarrow M'$ whose restriction to $M$ coincides with the canonical centre-valued trace on $M$.
\end{theorem}

The previous theorem relies on the study of the space $\wap_\Hh(M)$ of operators on $\Hh$ whose all coefficient functions are weakly almost periodic on $U_M$. The next section, which deals with an arbitrary von Neumann algebra $M\subset B(\Hh)$, is devoted to the following result and some consequences. See Theorem \ref{2.2} for a more detailed version.

\begin{theorem}
The set $\wap_\Hh(M)$ is a norm-closed, unital, selfadjoint subspace of $B(\Hh)$ which contains the commutant $M'$ of $M$ and the ideal $K(\Hh)$ of compact operators, and it is an $M'$-bimodule. In particular, it is spanned by its positive elements.
 Moreover, there exists a linear, positive, unital map $\E:\wap_\Hh(M)\rightarrow M'$ such that:
\begin{enumerate}
\item [(1)] For all $T\in \wap_\Hh(M)$ and $x',y'\in M'$, one has $\E(x'Ty')=x'\E(T)y'$.
\item [(2)] For every $T\in \wap_\Hh(M)$ and every $v\in U_M$, the operator $vTv^*$ belongs to $\wap_\Hh(M)$ and $\E(vTv^*)=\E(T)$.
\end{enumerate}
\end{theorem}

As will be seen in \S 2, the existence of $\E$ comes from the unique bi-invariant mean $\mf$ on $\WAP(U_M)$.

\begin{remark}
There is a priori no reason that $\wap_\Hh(M)$ be a $C^*$-algebra. However, the case of $\wap_\Hh(B(\Hh))$ is completely described as the following result shows.
\end{remark}

\begin{theorem}\label{1.4}
The operator system $\wap_\Hh(B(\Hh))$ is equal to
$K(\Hh)+\C$.
\end{theorem}

The penultimate section is devoted to the case of operators whose coefficient functions are almost periodic functions on $U_M$: recall that $f\in \cbr(G)$ is \textit{almost periodic} if the orbits $Gf$ and $fG$ are relatively compact for the norm topology. We denote by $\AP(G)$ the subset (in fact the $C^*$-subalgebra) of $\WAP(G)$ of all almost periodic functions on $G$. Analogously, we denote by $\ap_\Hh(M)$ the set of all operators $T\in B(\Hh)$ whose all coefficient functions are almost periodic. Then we prove the following result which relies on Theorem \ref{1.1}.

\begin{theorem}\label{1.5}
Let $M\subset B(\Hh)$ be a von Neumann algebra. Then $M$ is contained in $\ap_\Hh(M)$ if and only if $M$ is a direct sum $A\oplus\bigoplus_{k\geq 1} M_k$ where $A$ is an abelian, diffuse von Neumann algebra, and where each $M_k$ is a finite-dimensional factor. In other words, $M\subset \ap_\Hh(M)$ if and only if $M$ is a direct sum of von Neumann algebras each of which is either abelian or finite-dimensional.
\end{theorem}

A von Neumann algebra which is a direct sum of von Neumann algebras each of which is either abelian or finite-dimensional is a \textit{strongly finite} von Neumann algebra in the sense of \cite{GreenLau} where the authors prove that these algebras are characterized by the relative compactness in the norm topology of the orbits $\{\f\circ\alpha\colon \alpha\in\Int(M)\}$ for all normal linear functionals $\f\in M_*$. We are grateful to P. de la Harpe for having indicated that reference.

Theorem \ref{1.5} rests on the following theorem. See Theorem \ref{map} for a more precise statement.

\begin{theorem}
Let $M$ be a diffuse von Neumann algebra. Then its unitary group $U_M$ is minimally almost periodic, i.e. the only continuous, finite-dimensional, irreducible, unitary representation of $U_M$ is the trivial representation.
\end{theorem}

As explained previously, the last section is an appendix which is devoted to remind the reader of properties of weakly almost periodic functions on an arbitrary topological group $G$, and mostly the existence of the bi-invariant mean $\mf$ on $\WAP(G)$, which plays a central role in this article, as already mentioned.

\par\vspace{3mm}
\textit{Acknowledgements.} I am very grateful to P. de la Harpe for his numerous comments, questions, suggestions and references in the last versions of this article, and to the referee for his/her valuable comments.

\section{The operator system $\wap_\Hh(M)$}

Let $\Hh$ be a separable Hilbert space and let $M\subset B(\Hh)$ be a von Neumann algebra acting on $\Hh$. 

The next lemma contains general properties of coefficient functions, and in particular formulas for
the left and right translates by elements of $U_M$. Its proof uses straighforward computations which are left to the reader.

\begin{lemma}\label{1.2}
Let $T\in B(\Hh)$, $\xi,\eta\in\Hh$ and $v\in U_M$. Then the following formulas hold:
\begin{equation}\label{star}
\xi\star T^*\star\eta=\overline{\eta\star T\star\xi},
\end{equation}
\begin{equation}\label{gauche}
v\cdot(\xi\star T\star \eta)=(v\xi)\star T\star (v\eta)
\end{equation}
and
\begin{equation}\label{droite}
(\xi\star T\star \eta)\cdot v=\xi\star (vTv^*)\star \eta.
\end{equation}
Let furthermore $x',y'\in M'$. Then 
\begin{equation}\label{commutant}
\xi\star (x'Ty')\star\eta=(y'\xi)\star T\star(x'^*\eta).
\end{equation}
\end{lemma}

Our next lemma is essentially \cite[Lemme 1]{HarpeMoy}, but we provide a proof for the sake of completeness.

\begin{lemma}\label{1.3}
Let $T\in B(\Hh)$ and $\xi,\eta\in\Hh$. Then $\xi\star T\star \eta\in \cbr(U_M)$.
Moreover, $T\in M'$ if and only if all its associated coefficient functions are constant.

Furthermore, let $\f_{\xi,\eta}:U_M\rightarrow \C$ be defined by
\[
\f_{\xi,\eta}(u)=\langle u\xi|\eta\rangle\quad (u\in U_M).
\]
Then $\f_{\xi,\eta}\in \cbr(U_M)$.
\end{lemma} 
\begin{proof}
Let us set $\f:=\xi\star T\star \eta$ for short. Then one has, for every $u\in U_M$, by equality (\ref{gauche}):
\begin{align*}
|(v\cdot \f)(u)-\f(u)|
&=
|\langle uTu^*v\xi|v\eta\rangle-\langle uTu^*\xi|\eta\rangle|\\
&\leq
|\langle uTu^*(v\xi-\xi)|v\eta\rangle|+|\langle uTu^*\xi|v\eta-\eta\rangle|\\
&\leq
\Vert T\Vert\Vert v\xi-\xi\Vert\Vert\eta\Vert+\Vert T\Vert\Vert\xi\Vert\Vert v\eta-\eta\Vert.
\end{align*}
Hence
\[
\Vert v\cdot\f-\f\Vert_\infty \leq \Vert T\Vert\max(\Vert\xi\Vert,\Vert\eta\Vert)(\Vert v\xi-\xi\Vert+\Vert v\eta-\eta\Vert),
\]
which proves that $\Vert v\cdot\f-\f\Vert_\infty\to 0$ as $v\to 1$ in the strong operator topology.\\
The assertion relative to $M'$ is obvious.
\\
Finally, we have for every $v\in U_M$
\begin{align*}
\Vert v\cdot \f_{\xi,\eta}-\f_{\xi,\eta}\Vert_\infty
&=
\sup_{u\in U_M} |\langle v^*u\xi|\eta\rangle -\langle u\xi|\eta\rangle|\\
&\leq
\Vert \xi\Vert\Vert v\eta-\eta\Vert
\end{align*}
which proves that $\f_{\xi,\eta}\in \cbr(U_M)$.
\end{proof}

\begin{definition}\label{2.1}
A  linear, bounded operator $T\in B(\Hh)$ is \textit{almost periodic with respect to} $M$ (resp. \textit{weakly almost periodic with respect to} $M$) if, for all $\xi,\eta\in\Hh$, the coefficient function $\xi\star  T\star \eta$ belongs to $\AP(U_M)$ (resp. $\WAP(U_M)$). The set of all almost periodic operators with respect to $M$ is denoted by $\ap_\Hh(M)$, and similarly the set of all weakly almost periodic operators with respect to $M$ is denoted by $\wap_\Hh(M)$.
\end{definition}

We focus on $\wap_\Hh(M)$ because of the existence of the unique invariant mean $\mf$ on $\WAP(U_M)$, which implies the existence of a positive, $M'$-bimodular map $\E$ from $\wap_\Hh(M)$ onto $M'$. More precisely, one has the following result.

\begin{theorem}\label{2.2}
The set $\wap_\Hh(M)$ has the following properties: 
\begin{enumerate}
\item [(a)] It is a norm-closed, unital operator system in the sense of \cite[Chapter 2]{Paulsen}: it is a closed, selfadjoint subspace of $B(\Hh)$ which contains $1$, thus it is spanned by its positive elements.
\item [(b)] For every $T\in\wap_\Hh(M)$ and for all $x',y'\in M'$, one has $x'Ty'\in \wap_\Hh(M)$. In particular, $M'\subset \wap_\Hh(M)$.
\item [(c)] The ideal $K(\Hh)$ of all linear, compact operators on $\Hh$ is contained in $\wap_\Hh(M)$.
\end{enumerate}
Furthermore, there exists a linear, bounded and unital map 
\[\E:\wap_\Hh(M)\rightarrow B(\Hh)
\] 
which is characterized by the equality
\begin{equation}\label{defE}
\langle \E(T)\xi|\eta\rangle=\mf(\xi\star T\star \eta)\quad (\xi,\eta\in\Hh),
\end{equation}
and which possesses the following properties:
\begin{enumerate}
\item [(1)] $\E$ is a positive map.
\item [(2)] For every $T\in\wap_\Hh(M)$, one has $\E(T)\in M'$.
\item [(3)] For every $T\in\wap_\Hh(M)$ and all $x',y'\in M'$, one has $\E(x'Ty')=x'\E(T)y'$.  
\item [(4)] For every $T\in\wap_\Hh(M)$, $\E(T)$ belongs to $K_T$, where the latter denotes the weakly closed convex hull of the orbit $\{uTu^* \colon u\in U_M\}$.
\item [(5)] For every $T\in \wap_\Hh(M)$ and every $v\in U_M$, the operator $vTv^*$ belongs to $\wap_\Hh(M)$, and $\E(vTv^*)=\E(T)$.
\item [(6)] For every $C^*$-subalgebra $A$ of $\wap_\Hh(M)$, the restriction of $\E$ to $A$ is completely positive.
\end{enumerate}
\end{theorem}
\begin{proof}
(a) As $\WAP(U_M)$ is a $C^*$-subalgebra of $\cbr(U_M)$, $\wap_\Hh(M)$ is a norm-closed subspace of $B(\Hh)$ because, if $\Vert T_n-T\Vert\to_{n\to\infty} 0$
then 
\[
\Vert \xi\star T_n\star \eta-\xi\star T\star \eta\Vert_\infty\leq \Vert T_n-T\Vert\Vert\xi\Vert\Vert\eta\Vert\to_{n\to\infty}0.
\]
The fact that $T^*\in\wap_\Hh(M)$ for every $T\in\wap_\Hh(M)$ follows from equation (\ref{star}) of Lemma \ref{1.2}. Hence $\wap_\Hh(M)$ is spanned by its selfadjoint elements, and, as $1\in\wap_\Hh(M)$, if $T=T^*\in \wap_\Hh(M)$, then the equality 
\[
T=\frac{1}{2}(\Vert T\Vert+T)-\frac{1}{2}(\Vert T\Vert-T)
\]
shows that $\wap_\Hh(M)$ is spanned by its positive elements.
\\
(b) We have $M'\subset \wap_\Hh(M)$ thanks to Lemma \ref{1.3}. Moreover, $\wap_\Hh(M)$ is an $M'$-bimodule by equation (\ref{commutant}) of Lemma \ref{1.2}. 
\\
(c) As $\wap_\Hh(M)$ is a closed subspace of $B(\Hh)$, in order to show that $K(\Hh)\subset\wap_\Hh(M)$, it suffices to prove that $\wap_\Hh(M)$ contains all rank one operators. Thus, let $\zeta,\om\in\Hh$ and let $T_{\zeta,\om}$ be the rank one operator defined by
\[
T_{\zeta,\om}(\xi):=\langle \xi|\om\rangle\zeta\quad (\xi\in\Hh).
\]
Then we have for all $\xi,\eta\in \Hh$ and $u,v\in U_M$:
\begin{align*}
v\cdot(\xi\star T_{\zeta,\om}\star \eta)(u)
&=
(\xi\star T_{\zeta,\om}\star \eta)(v^*u)=\langle v^*uT_{\zeta,\om}(u^*v\xi)|\eta\rangle\\
&=
\langle v^*u\langle u^*v\xi|\om\rangle\zeta|\eta\rangle=\langle u^*v\xi|\om\rangle\,\langle v^*u\zeta|\eta\rangle\\
&=
\overline{\langle v^*u\om|\xi\rangle}\,\langle v^*u\zeta|\eta\rangle\\
&=
v\cdot(\overline{\f_{\om,\xi}}\f_{\zeta,\eta})(u)
\end{align*}
where we set $\f_{\xi,\eta}:u\mapsto \langle u\xi|\eta\rangle$ for all $\xi,\eta\in\Hh$. As $\WAP(U_M)$ is a $*$-algebra, in order to show that the left orbit $U_M(\xi\star T_{\zeta,\om}\star \eta)$ is relatively weakly compact, it suffices to prove that, for all $\xi,\eta\in\Hh$, the orbit $U_M\f_{\xi,\eta}$ is relatively weakly compact. Thus, let us fix $\xi,\eta\in\Hh$. One has for all $u,v\in U_M$:
\[
(v\cdot \f_{\xi,\eta})(u)=\langle v^*u\xi|\eta\rangle=\langle u\xi|v\eta\rangle=\f_{\xi,v\eta}(u).
\]
Since $\f_{\xi,\eta}\in \cbr(U_M)$ by Lemma \ref{1.3} and since
the orbit $\{v\eta \colon v\in U_M\}$ is weakly relatively compact in $\Hh$, it suffices to prove that the map $\eta\mapsto \f_{\xi,\eta}$ is continuous when $\Hh$ and $\cbr(U_M)$ are equipped with their respective weak topologies. Thus, let $\mu$ be a continuous linear functional on $\cbr(U_M)$. The sesquilinear form
$(\zeta,\om)\mapsto \mu(\f_{\zeta,\om})$ satisfies the following inequality: $|\mu(\f_{\zeta,\om})|\leq \Vert\mu\Vert\Vert\zeta\Vert\Vert\om\Vert$. Hence there exists a unique operator $T_\mu\in B(\Hh)$ such that $\mu(\f_{\zeta,\om})=\langle T_\mu\zeta|\om\rangle$ for all $\zeta,\om\in\Hh$.

Now, if $(\eta_n)\subset \Hh$ converges weakly to $\eta$, we have
\[
\mu(\f_{\xi,\eta_n})=\langle T_\mu\xi|\eta_n\rangle\to_{n\to\infty}
\langle T_\mu\xi|\eta\rangle=\mu(\f_{\xi,\eta}).
\]
This ends the proof of the fact that all rank one operators (hence all compact operators) belong to $\wap_\Hh(M)$.

Let us now prove the existence of the map $\E$ and all its stated properties.
For $T\in\wap_\Hh(M)$, the sesquilinear form $(\xi,\eta)\mapsto \mf(\xi\star T\star \eta)$ is continuous since one has $\Vert\xi\star T\star \eta\Vert_\infty \leq \Vert T\Vert\Vert\xi\Vert\Vert\eta\Vert$. Hence this proves the existence and uniqueness of $\E(T)$ for every $T\in\wap_\Hh(M)$, as well as its linearity and boundedness.
\\
(1) If $T\in \wap_\Hh(M)$ is a positive operator and if $\xi\in\Hh$, then 
\[
\xi\star T\star \xi(u)=\langle uTu^*\xi|\xi\rangle=\langle Tu^*\xi|u^*\xi\rangle\geq 0,
\] 
which implies that $\E(T)\geq 0$ since $\mf$ is a positive functional.
\\
(2) Let $T\in\wap_\Hh(M)$, $v\in U_M$ and $\xi,\eta\in\Hh$. Then, using equality (\ref{gauche}) and left invariance of $\mf$, we get
\begin{align*}
\langle v^*\E(T)v\xi|\eta\rangle
&=
\langle \E(T)v\xi|v\eta\rangle=\mf((v\xi)\star T\star (v\eta))\\
&=
\mf(v\cdot (\xi\star T\star \eta))=\mf(\xi\star T\star \eta)\\
&=
\langle \E(T)\xi|\eta\rangle,
\end{align*}
which shows that $v^*\E(T)v=\E(T)$ for every $v\in U_M$, thus $\E(T)\in M'$.
\\
(3) follows from equality \ref{commutant}.
\\
(4) We could reproduce the proof of statement (iii) of \cite[Lemme 2]{HarpeMoy}, but we present a different one, based on the following property of the mean $\mf$ (property (d) of Theorem \ref{moy}): For every weakly almost periodic function $f$ on $U_M$, its mean $\mf(f)$ belongs to the norm-closed convex hull  of its right orbit $fU_M$. Thus, if $\xi_1,\ldots,\xi_n,\eta_1,\ldots,\eta_n\in\Hh$ and $\ep>0$ are given, there exist $s_1,\ldots,s_m>0$, $\sum_i s_i=1$, and $v_1,\ldots,v_m\in U_M$ such that
\begin{equation}\label{convexe}
\Big\Vert \mf\Big(\sum_j \xi_j\star T\star \eta_j\Big)-\sum_i s_i\Big(\sum_j \xi_j\star T\star \eta_j\Big)\cdot v_i\Big\Vert_\infty\leq \ep.
\end{equation}
By equality (\ref{droite}), one has
\[
\sum_i s_i\Big(\sum_j \xi_j\star T\star \eta_j\Big)\cdot v_i=\sum_i s_i\Big(\sum_j \xi_j\star (v_iTv_i^*)\star \eta_j\Big),
\]
which yields, when evaluated at $u=1$, 
\begin{align*}
\sum_i s_i\Big(\sum_j \xi_j\star (v_iTv_i^*)\star \eta_j\Big)(1)
&=
\sum_i s_i\Big(\sum_j \langle v_iTv_i^*\xi_j|\eta_j\rangle\Big)\\
&=
\sum_j \langle\Big(\sum_i s_iv_iTv_i^*\Big)\xi_j|\eta_j\rangle.
\end{align*}
As $\mf\Big(\sum_j \xi_j\star T\star \eta_j\Big)=\sum_j \langle \E(T)\xi_j|\eta_j\rangle$, we get
\[
\left|\sum_j \langle \E(T)\xi_j|\eta_j\rangle-\sum_j \langle\Big(\sum_i s_iv_iTv_i^*\Big)\xi_j|\eta_j\rangle\right|\leq \ep
\]
which proves the claim.
\\
(5) For $\xi,\eta\in \Hh$, equality (\ref{droite}) shows that the right orbit $(\xi\star (vTv^*)\star \eta)U_M=((\xi\star T\star \eta)\cdot v)U_M$ is weakly relatively compact, hence that $vTv^*\in \wap_\Hh(M)$. As $\mf$ is right invariant, we get
\[
\langle\E(vTv^*)\xi|\eta\rangle=\mf((\xi\star T\star \eta)\cdot v)=\mf(\xi\star T\star \eta)=\langle\E(T)\xi|\eta\rangle
\]
which proves (5). 
\\
(6) By \cite[Corollary 3.4, Chapter IV]{Tak1}, it suffices to prove that 
\[
\sum_{i,j}\Big\langle y_i'^*\E(a_i^*a_j)y_i'\xi|\xi\Big\rangle\geq 0
\]
for all $a_1,\ldots,a_n\in A$, all $y_1',\ldots,y_n'\in M'$ and every $\xi\in\Hh$. We have
\begin{align*}
\sum_{i,j}\Big\langle y_i'^*\E(a_i^*a_j)y_i'\xi|\xi\Big\rangle
&=
\sum_{i,j}\Big\langle \E([a_iy_i']^*[a_jy_j'])\xi|\xi\Big\rangle\\
&=
\mf\Big(\sum_{i,j}\xi*[(a_iy_i')^*(a_jy_j')]*\xi\Big)\geq 0
\end{align*}
because 
\[
\left(\sum_{i,j}\xi*[(a_iy_i')^*(a_jy_j')]*\xi\right)(u)\geq 0
\]
for every $u\in U_M$.
Indeed, set
\[
C=
\left(
\begin{array}{cccc}
a_1y_1' & 0 & \ldots & 0\\
a_2y_2' & 0 & \ldots & 0 \\
\vdots & & & \vdots\\
a_ny_n' & 0 & \ldots & 0
\end{array}\right)\in M_n(B(\Hh)).
\]
Then 
\[
C^*C=
\left(
\begin{array}{cccc}
\sum_{i,j}(a_iy_i')^*(a_jy_j') & 0 & \ldots & 0\\
0 & 0 & \ldots & 0 \\
\vdots & & & \vdots\\
0 & 0 & \ldots & 0
\end{array}\right),
\]
and if 
\[
\zeta_u=\left(
\begin{array}{c}
u^*\xi\\
0\\
\vdots\\
0
\end{array}\right),
\]
we have
\[
\left(\sum_{i,j}\xi*[(a_iy_i')^*(a_jy_j')]*\xi\right)(u)=\langle C^*C\zeta_u|\zeta_u\rangle\geq 0
\]
for every $u\in U_M$.
\end{proof}

The following result describes completely the space $\wap_\Hh(B(\Hh))$. 

\begin{theorem}\label{2.6}
We have the following equality:
\[
\wap_\Hh(B(\Hh))=K(\Hh)+\C.
\]
In particular, if $\wap_\Hh(B(\Hh))=B(\Hh)$ then $\Hh$ is finite-dimensional.
\end{theorem}
\begin{proof}
Since $\Hh$ is a separable Hilbert space, it follows from a result due to Weyl \cite[Proposition 4, p.194]{HarpeCalkin} that 
every selfadjoint operator $T\in B(\Hh)$ is a sum $T=D+K$ where
$K=K^*\in K(\Hh)$ and where $D$ is a selfadjoint, diagonal operator, \textit{i.e.} there exists an orthonormal basis $(\ep_i)_{i\in I}$ of $\Hh$ and real numbers $(\lambda_i)_{i\in I}$ such that $D\ep_i=\lambda_i\ep_i$ for every $i\in I$.

Hence, it suffices to prove that if $D=D^*\in \wap_\Hh(B(\Hh))$ then $D=D_c+\lm$ where $D_c$ is a selfadjoint, compact operator and $\lm\in\R$. We assume that $\Hh$ is infinite-dimensional and we set $U(\Hh):=U_{B(\Hh)}$. The proof is divided into three steps.\\
(i) Let $D=D^*\in \wap_\Hh(B(\Hh))$ be a selfadjoint, diagonal operator. Then $D$ has at most one eigenvalue whose associated eigenspace is infinite-dimensio-nal.
Indeed, assume on the contrary that $D=D^*$ has two eigenvalues $\lm_1\not=\lm_2$ whose associated eigenspaces $\Hh_1$ and $\Hh_2$ are infinite-dimensional. We claim that $D$ does not belong to $\wap_\Hh(B(\Hh))$. In order to prove that, we are going to apply Proposition \ref{criterewap}. Thus, let $(\ep_j)_{j\geq 1}$ (resp. $(\delta_j)_{j\geq 1}$) be an orthonormal basis of $\Hh_1$ (resp. $\Hh_2$). For all integers $i,j\geq 1$, we define selfadjoint, unitary operators $u_i$ and $v_j$ as follows: their restrictions to the orthogonal complement $(\Hh_1\oplus\Hh_2)^\perp$ is the identity, and then set
\[
u_i\ep_k=
\begin{cases}
\ep_k, & k\leq i\\
\delta_k, & k>i
\end{cases}
\]
and 
\[
u_i\delta_k=
\begin{cases}
\delta_k, & k\leq i\\
\ep_k, & k>i.
\end{cases}
\]
Next,
$v_j$ exchanges $\ep_1$ and $\ep_{j}$, and $v_j\xi=\xi$ for all $\xi\in\{\ep_1,\ep_{j}\}^\perp$. 

Let us fix $i\geq 1$; we have
\[
u_iv_j\ep_1=u_i\ep_{j}=\delta_{j}
\]
for $j>i$, hence 
\[
\ep_1\star D\star\ep_1(v_ju_i)=\langle v_ju_iDu_iv_j\ep_1|\ep_1\rangle=
\langle D\delta_{j}|\delta_{j}\rangle=\lm_2
\]
for every $j>i$, and thus $\lim_j\ep_1\star D\star\ep_1(v_ju_i)=\lm_2$ for every $i$. This gives
\[
\lim_i(\lim_j\ep_1\star D\star\ep_1(v_ju_i))=\lm_2.
\]
Next, let us fix $j\geq 1$. If $i>j$, one has 
\[
u_iv_j\ep_1=u_i\ep_{j}=\ep_{j},
\]
and we get $\lim_i\ep_1\star D\star\ep_1(v_ju_i)=\lm_1$. This implies that 
\[
\lim_j(\lim_i\ep_1\star D\star\ep_1(v_ju_i))\not=\lim_i(\lim_j\ep_1\star D\star\ep_1(v_ju_i)),
\]
and Proposition \ref{criterewap} implies that $\ep_1\star D\star \ep_1\notin \WAP(U(\Hh))$, which proves the claim. 
In particular, if $D$ has finite spectrum, then it has exactly one eigenvalue $\lm$ whose associated eigenspace is infinite-dimensional. This implies that $D-\lm$ is a finite-rank operator.\\
(ii) Assume next that $D$ admits infinitely many distinct eigenvalues $(\lm_k)_{k\geq 1}$ such that at most one of them, say $\lm_1$, has an infinite-dimensional eigenspace. If it is the case, replacing $D$ by $D-\lm_1$, we assume that all eigenspaces of all non-zero eigenvalues are finite-dimensional. As the spectrum of $D$ is infinite and bounded, it possesses at least one accumulation point. In fact, we claim that it possesses exactly one such point. Indeed, if $(\lm_k)$ has two distinct accumulation points $\alpha\not=\beta$, then there are sequences $(u_i),(v_j)\subset U(\Hh)$ and a vector $\xi\in\Hh$ such that
\[
\lim_i(\lim_j \xi\star D\star\xi(v_ju_i))=\alpha\not=\beta=\lim_j(\lim_i \xi\star D\star\xi(v_ju_i)),
\]
and $D$ does not belong to $\wap_\Hh(B(\Hh))$. Indeed, let $(\alpha_k)_{k\geq 1}$ (resp. $(\beta_k)_{k\geq 1}$) be a sequence of distincts eigenvalues of $D$ which converges to $\alpha$ (resp. $\beta$), and such that $\alpha_k\not=\beta_\ell$ for all $k,\ell$. Choose for every $k$ a norm-one eigenvector $\ep_k$ (resp. $\delta_k$) of $\alpha_k$ (resp. $\beta_k$), so that the sequences $(\ep_k)$ and $(\delta_k)$ are orthonormal. Then, for all $i,j\geq 1$, define $u_i$ and $v_j$ exactly as in Part (i). Then for every $i$, one has
$\ep_1\star D\star \ep_1(v_ju_i)=\beta_{j}\to \beta$ as $j\to\infty$, so that
\[
\lim_i(\lim_j \ep_1\star D\star\ep_1(v_ju_i))=\beta
\]
and similarly,
\[
\lim_j(\lim_i \ep_1\star D\star\ep_1(v_ju_i))=\alpha,
\]
which proves the claim.\\
(iii) Finally, we are left with the case where the set of distinct non-zero eigenvalues of $D$ is infinite and has a unique accumulation point $\lm$, and such that every eigenspace is finite-dimensional (except maybe the one associated to $\lm$). Thus $D-\lm$ has the same properties as $D$, but as $0$ is the only accumulation point of $D-\lm$ and as all non-zero eigenvalues have finite-dimensional eigenspaces, this proves that $D-\lm$ is a compact operator.
\end{proof}

\begin{proposition}\label{2.5}
Assume that $\Hh$ is infinite-dimensional, and let $M\subset B(\Hh)$ be either a diffuse von Neumann algebra, or $M=B(\Hh)$, and let $\E:\wap_\Hh(M)\rightarrow M'$ be the positive map of Theorem \ref{2.2}. Then $K(\Hh)\subset \ker\E$.
\end{proposition}
\begin{proof}
Since $\E$ is a continuous, positive map, it suffices to prove that, for all $\zeta,\xi\in\Hh$, $\langle \E(T_{\zeta,\zeta})\xi|\xi\rangle=0$  where, as in the proof of Theorem \ref{2.2}, $T_{\zeta,\zeta}$ is the rank one operator defined by $T_{\zeta,\zeta}(\xi)=\langle \xi|\zeta\rangle\zeta$ for every $\xi\in\Hh$.

Thus, fix $\zeta,\xi\in \Hh$; we have for every $u\in U_M$:
\begin{align*}
\xi\star T_{\zeta,\zeta}\star\xi(u)
&=
\langle uT_{\zeta,\zeta}u^*\xi|\xi\rangle
=\langle \langle u^*\xi|\zeta\rangle u\zeta|\xi\rangle\\
&=
\langle \xi|u\zeta\rangle\langle u\zeta|\xi\rangle=|\langle u\zeta|\xi\rangle|^2.
\end{align*}
The function $\f$ on $U_M$ defined by $\f(u):=\langle u\zeta|\xi\rangle$ for every $u\in U_M$ belongs to $\WAP(U_M)$ by the proof of property (c) of Theorem \ref{2.2}. As 
\[
|\langle u\zeta|\xi\rangle|^2\leq \Vert\zeta\Vert\Vert\xi\Vert |\langle u\zeta|\xi\rangle|
\] 
for every $u\in U_M$, it suffices to prove that $\mf(|\f|)=0$. Its proof is inspired by that of \cite[Theorem 1.3]{BR}. Let $\ep>0$. By condition (d) of Theorem \ref{moy}, there exist $v_1,\ldots,v_m\in U_M$ and $t_1,\ldots,t_m>0$ such that $\sum_j t_j=1$ and
\[
\Big|\sum_j t_j |\f(v_j^*u)|-\mf(|\f|)\Big|<\ep/2
\]
for every $u\in U_M$. Assume first that $M$ is diffuse; then there exists a sequence $(u_n)\subset U_M$ such that $u_n\to 0$ weakly. Thus, there exists $n$ such that 
$|\f(v_j^*u_n)|<\ep/2$ for every $j$, so that 
\[
0\leq \sum_j t_j |\f(v_j^*u_n)|<\ep/2.
\]
This implies that
\begin{align*}
0\leq \mf(|\f|)
&\leq
\Big|\sum_j t_j |\f(v_j^*u_n)|-\mf(|\f|)\Big|+\sum_j t_j |\f(v_j^*u_n)|\\
&\leq 
\ep.
\end{align*}
Finally, assume that $M=B(\Hh)$, and let $F\subset \Hh$ be the linear span of the set $\{v_1\xi,\ldots,v_m\xi\}$, which is finite-dimensional. Since $\Hh$ is infinite-dimensional, there exists $\eta\in F^\perp$ such that $\Vert\eta\Vert=\Vert\zeta\Vert$. Hence, there exists a unitary operator $u$ on $\Hh$ such that $u\zeta=\eta\perp v_j\xi$ for every $j=1,\ldots,m$. This implies that $\f(v_j^*u)=0$ for every $j$ and thus that $\mf(|\f|)<\ep/2$.
\end{proof}

We end the present section with three remarks. The last one is inspired by \cite[Remarques (i)]{HarpeMoy}.

\begin{remark}
Let $\mu\in \cbr(U_M)^*$ be a continuous linear functional on the $C^*$-algebra $\cbr(U_M)$. Then, the effect of $\mu$ on $B(\Hh)$ can be described as follows:
as in the proof of Theorem \ref{2.2}, there exists a unique bounded, linear map $\E_\mu:B(\Hh)\rightarrow B(\Hh)$ such that 
\[
\langle \E_\mu(T)\xi|\eta\rangle=\mu(\xi\star T\star\eta)
\]
for all $T\in B(\Hh)$ and $\xi,\eta\in\Hh$. Equation (\ref{commutant}) implies that $\E_\mu$ is an $M'$-bimodular map, i.e. $\E_\mu(x'Ty')=x'\E_\mu(T)y'$ for all $T\in B(\Hh)$ and $x',y'\in M'$.
\end{remark}

\begin{remark}
If $M$ is such that $\wap_\Hh(M)=B(\Hh)$, then $M$ is Approximately Finite-Dimensional. Indeed, if it is the case, as in \cite{HarpeMoy}, the map $\E$ is a conditional expectation from $B(\Hh)$ onto $M'$, and $M$ has Schwartz's property P. Moreover, as will be proved in \S 3 (Theorem \ref{3.1} and Remark \ref{injective}), $M$ is finite since it is then contained in $\wap_\Hh(M)$.
\end{remark}

\begin{remark}
Let $\Hh$ be an infinite-dimensional Hilbert space and let $A\subset B(\Hh)$ be a unital $C^*$-algebra which has no tracial states. Let $M\subset B(\Hh)$ be a von Neumann algebra containing $A$. Then $A\not\subset \wap_\Hh(M)$. Indeed, otherwise, we would have $\E(uau^*)=\E(a)$ for all $a\in A$ and $u\in U_M$. This would imply that $\E(xy-yx)=0$ for all $x,y\in A$, but this leads to a contradiction since $1$ is a finite sum of commutators of elements of $A$ by \cite[Theorem 1]{Pop}.
\end{remark}

\section{The case of finite von Neumann algebras}

In this section, we consider a von Neumann algebra $M\subset B(\Hh)$ and we denote by $M_*$ its predual. We denote by $\Int(M)$ the group of its inner automorphisms and, for $v\in U_M$, we denote by $\Ad(v)\in\Int(M)$ the automorphism given by $\Ad(v)(x)=vxv^*$ for every $x\in M$.

Let $B(M)$ (resp. $B_*(M)$) denote the Banach space of all bounded (resp. ultraweakly continuous) linear operators on $M$. The weak$^*$ topology on $B(M)$ is the $\sigma(B(M),M\otimes_\gamma M_*)$-topology, where $M\otimes_\gamma M_*$ is the projective tensor product of $M$ and $M_*$ (see \cite[Chapter IV]{Tak1}).
In fact, if $(\Phi_i)\subset B(M)$ is a bounded net, then it converges weakly$^*$ to $\Phi\in B(M)$ if and only if 
\[
\f(\Phi_i(x))=\langle \Phi_i,x\otimes\f\rangle \to \f(\Phi(x))=\langle \Phi,x\otimes\f\rangle
\]
for all $x\in M$ and $\f\in M_*$.

If $M$ is finite, we denote by $\ctr_M$ its canonical centre-valued trace.  
In this case, it follows from (the proof of) \cite[Theorem V.2.4]{Tak1} that the weak$^*$ closure of $\Int(M)$ in $B(M)$ is contained in $B_*(M)$. Thus, as $\Int(M)$ is bounded, it is relatively weakly$^*$ compact in $B_*(M)$ (see \cite[Exercise 6, p. 333]{Tak1}.)

This allows us to prove the following theorem.

\begin{theorem}\label{3.1}
The von Neumann algebra $M$ is finite if and only if $M\subset \wap_\Hh(M)$. In other words, $M$ is finite if and only if, for every $x\in M$, all its coefficient functions $\xi\star x\star \eta$ are weakly almost periodic.
If it is the case, $\wap_\Hh(M)$ contains $C^*(M,M')$, the $C^*$-algebra generated by $M$ and $M'$ in $B(\Hh)$, and the restriction $\E_M$ of $\E$ to $C^*(M,M')$ has the additional properties:
\begin{enumerate}
\item [(i)] The restriction of $\E_M$ to $M$ coincides with the canonical centre-valued trace $\ctr_M$.
\item [(ii)] $\E_M:C^*(M,M')\rightarrow M'$ is a conditional expectation.
\end{enumerate}
\end{theorem}
\begin{proof}
If $M\subset \wap_\Hh(M)$, then the restriction of $\E$ to $C^*(M,M')$ satisfies condition (5) in Theorem \ref{2.2}, and this implies that $\E(xy)=\E(yx)$ for all $x,y\in M$. Furthermore, if $x\in M$ and $u'\in U_{M'}$, one has by property (3) in Theorem \ref{2.2}: $u'^*\E(x)u'=\E(u'^*xu')=\E(x)$, which means that the restriction of $\E$ to $M$ maps $M$ onto its centre $Z(M)$. Then \cite[Corollary 3, Part III, Chapter 8]{DivN} shows that $M$ is a finite von Neumann algebra.

Conversely, suppose that $M$ is finite.
Let $x\in M$ and $\xi,\eta\in \Hh$. We must prove that the coefficient function $\xi\star x\star \eta$ belongs to $\WAP(U_M)$.
It follows from \cite[Proposition 7]{Grothen} or by Proposition \ref{4.3} of the appendix that it suffices to prove that the right orbit $(\xi\star x\star   \eta)U_M$ is relatively weakly compact in $\cbr(U_M)$. But we have, by equation (\ref{droite}):
\[
(\xi\star x\star \eta)U_M=\{\xi\star vxv^*\star \eta\colon v\in U_M\}=\{\xi\star \theta(x)\star \eta\colon \theta\in \Int(M)\}.
\]
Thus, from Grothendieck's Theorem \ref{Grothendieck} in the appendix, it suffices to prove that if $(u_i)\subset U_M$ and $(\theta_j)\subset \Int(M)$ are sequences such that the following double limits
\[
\ell_1:=\lim_i(\lim_j(\xi\star \theta_j(x)\star \eta)(u_i))
\quad\textrm{and}\quad
\ell_2:=\lim_j(\lim_i(\xi\star \theta_j(x)\star \eta)(u_i))
\]
exist, then they are equal. Set $\psi_i=\Ad(u_i)$ for every $i$. Then $(\xi\star \theta_j(x)\star \eta)(u_i)=\langle u_i\theta_j(x)u_i^*\xi|\eta\rangle=
\langle \psi_i(\theta_j(x))\xi|\eta\rangle$ for all $i,j$. Extracting subsequences if necessary, we assume that $(\theta_j)$ converges weakly$^*$ to some limit $\theta\in \overline{\Int(M)}\subset B_*(M)$, and that $(\psi_i)$ converges weakly$^*$ to some limit $\psi\in \overline{\Int(M)}\subset B_*(M)$. Denoting by $\om_{\xi,\eta}\in M_*$ the normal linear functional $y\mapsto \langle y\xi|\eta\rangle$, 
one has for every $i$, on the one hand, 
\[
\lim_j\langle \psi_i(\theta_j(x))\xi|\eta\rangle=\langle \psi_i(\theta(x))\xi|\eta\rangle=\langle \psi_i\circ\theta, x\otimes \om_{\xi,\eta}\rangle
=\langle \psi_i,\theta(x)\otimes\om_{\xi,\eta}\rangle
\]
because $\psi_i$ is a normal map. 
\\
Hence, $\ell_1=\lim_i\langle \psi_i, \theta(x)\otimes \om_{\xi,\eta}\rangle=\langle\psi\circ\theta(x)\xi|\eta\rangle$.

On the other hand, one has for every $j$
\[
\lim_i\langle \psi_i(\theta_j(x))\xi|\eta\rangle=\langle\psi(\theta_j(x))\xi|\eta\rangle.
\]
As $\psi$ is normal, we get
\[
\ell_2=\lim_j\langle\psi(\theta_j(x))\xi|\eta\rangle=\langle\psi\circ\theta(x)\xi|\eta\rangle=\ell_1.
\]
Hence $M\subset \wap_\Hh(M)$.

From now on, we assume that $M$ is finite.
\\
(i) For every $x\in M$, the map $\E$ is constant on the convex hull of $\{uxu^*\colon u\in U_M\}$, hence, by its boundedness, it is constant on the norm closure $K'_x$ of the latter. By \cite[Theorem 1, Part III, Chapter 5]{DivN}, $\E(\ctr_M(x))=\E(x)$, and as $\ctr_M(x)\in M\cap M'$, we have $\E(x)=\ctr_M(x)$. This proves (i).
\\
(ii) follows from property (3) of Theorem \ref{2.2}.
\end{proof}

\begin{remark}
The relative compactness of $\Int(M)$ in $B_*(M)$ is a special case of the notion of $G$-finite von Neumann as defined in \cite{KS}; see also \cite{Stoermer} and \cite{Yeadon}.
\end{remark}

The following example shows that the hypothesis that $M$ is diffuse in Proposition \ref{2.5} cannot be removed. Even though it is a special case of Theorem \ref{3.1}, we think it is worth being discussed.

\begin{example}
Set $\Hh=\ell^2(\N)$, let $(\delta_k)_{k\in\N}$ be the natural orthonormal basis of $\Hh$ and let $M=A=\ell^\infty(\N)$ be the atomic maximal abelian $*$-subalgebra of $B(\Hh)$ acting by pointwise multiplication on $\Hh$ so that $a\delta_k=a(k)\delta_k$ for all $a\in A$ and $k\in\N$. We claim that $\wap_\Hh(A)=B(\Hh)$ and that $\E(T)\in A$ is the function $k\mapsto \langle T\delta_k|\delta_k\rangle$ for every $T\in B(\Hh)$. 
\\
Indeed, let $T\in B(\Hh)$. In order to prove that it belongs to $\wap_\Hh(A)$, it suffices to verify that $\delta_k\star T\star\delta_\ell$ is weakly almost periodic for all $k,\ell\in\N$. Thus, let us fix integers $k$ and $\ell$. As $U_A=\T^\N$, where $\T$ denotes the unit circle, we have for all $u\in U_A$:
\[
\delta_k\star T\star\delta_\ell(u)=\langle T\overline{u(k)}\delta_k|\overline{u(\ell)}\delta_\ell\rangle=\langle T\delta_k|\delta_\ell\rangle \overline{u(k)}u(\ell).
\]
But
\begin{align*}
\overline{u(k)}u(\ell)
&=
\langle \delta_k|u\delta_k\rangle\langle u\delta_\ell|\delta_\ell\rangle=\langle \langle \delta_k|u\delta_k\rangle u\delta_\ell|\delta_\ell\rangle\\
&=
\langle u\langle u^*\delta_k|\delta_k\rangle\delta_\ell|\delta_\ell\rangle=\langle uT_{\delta_\ell,\delta_k}u^*\delta_k|\delta_\ell\rangle\\
&=
\delta_k\star T_{\delta_\ell,\delta_k}\star \delta_\ell(u)
\end{align*}
where we use the same notation as in the proof of Theorem \ref{2.2}(c) for rank one operators. Hence $T\in\wap_\Hh(A)$ and
\[
\langle \E(T)\delta_k|\delta_\ell\rangle=\langle T\delta_k|\delta_\ell\rangle\cdot \mf(\delta_k\star T_{\delta_\ell,\delta_k}\star \delta_\ell).
\]
Set $\f_{k,\ell}=\delta_k\star T_{\delta_\ell,\delta_k}\star \delta_\ell$ for short so that $\f_{k,\ell}(u)=\overline{u(k)}u(\ell)$ for every $u\in U_A$.

If $k=\ell$, then $\f_{k,k}(u)=1$ for every $u$ and $\mf(\f_{k,k})=1$. If $k\not=\ell$, then we can view $\f_{k,\ell}$ as the continuous function on $\T^2$
defined by $\f_{k,\ell}(z,w)=\overline zw$. As $\T^2$ is a compact group, one has $C(\T^2)=\WAP(\T^2)$ and the invariant mean on the latter coincides with the Haar measure. Hence, by property (d) of Theorem \ref{moy}, we have
\[
\mf(\f_{k,\ell})=\iint_{\T^2}\overline zwdzdw=0.
\]
\end{example}

\begin{remark}\label{injective}
Assume that $M$ is a finite von Neumann subalgebra of $B(\Hh)$, and let $\tau$ be a normal, faithful normalized trace on $M$. Set $\Vert x\Vert_2:=\tau(x^*x)^{1/2}$ for every $x\in M$. Then the topology on $U_M$ is induced by the complete, bi-invariant metric $(u,v)\mapsto \Vert u-v\Vert_2$. This means that every bounded, continuous function $f$ on $U_M$ is uniformly right continuous if and only if it is uniformly left continuous, namely that the map $v\mapsto v\cdot f$ is continuous if and only if the map $v\mapsto f\cdot v$ is. Suppose moreover that $M$ is injective, and let $(M_n)_{n\geq 1}$ be an increasing sequence of finite-dimensional $C^*$-subalgebras of $M$, such that $1\in M_n$ for every $n$ and that their union $M_\infty:=\bigcup_n M_n$ is strongly dense in $M$. Then, as proved in \cite{HarpeMoy}, $\cbr(U_{M_\infty})$ has a bi-invariant mean $\mu$, hence so does $\cbr(U_M)$ by the bilateral uniform continuity of all elements of the latter algebra. Consequently, even if we do not know whether $\wap_\Hh(M)$ is equal to $B(\Hh)$, the existence of $\mu$ implies the existence of a positive map $\E_\mu:B(\Hh)\rightarrow M'$ which has properties (1)--(6) of Theorem \ref{2.2} as well as those in Theorem \ref{3.1}. In particular, the restriction of $\E_\mu$ to $M$ is equal to $\ctr_M$.
\end{remark}


\section{The case of almost periodic operators relative to $M$}

Let $M\subset B(\Hh)$ be a von Neumann algebra. Recall that $\ap_\Hh(M)$ is the set of all operators $T\in B(\Hh)$ such that $\xi\star T\star\eta\in \AP(U_M)$ for all $\xi,\eta\in\Hh$. As in the proof of Theorem \ref{2.2}, $\ap_\Hh(M)$ is an operator system which contains $M'$. We remind the reader of characterizations of almost periodic functions on topological groups in Theorem \ref{apgroup}, which is taken from \cite[Theorem 16.2.1]{DiC}. We are grateful to P. de la Harpe for having indicated that reference.

We also need to recall definitions of diffuse and atomic von Neumann algebras. Our reference on these notions is \cite{SS}.
Denote by $\Pp(M)$ the set of all orthogonal projections of $M$. An element $e\in \Pp(M)$ is an \textit{atom} in $M$ if it satisfies the equality: $eMe=\C e$. If it is the case, then its central cover $z(e)$ (i.e. the smallest projection $z$ of the centre $Z(M)$ of $M$ such that $ze=e$) is an atom of the centre $Z(M)$, thus $Mz(e)$ is a factor.

\begin{definition}
The von Neumann algebra $M$ is called \textit{atomic} if, for every non-zero projection $f\in\Pp(M)$, there exists an atom $e\in \Pp(M)$ such that $e\leq f$.
If $M$ contains no atoms, then it is called \textit{diffuse}. 
\end{definition}

As a consequence of the above facts, if $M$ is atomic and finite, then it is a direct sum of finite-dimensional factors.

\begin{lemma}\label{diffuse}
Let $M$ be a von Neumann algebra. Then there is a unique central projection $z$ such that $Mz$ is diffuse and $M(1-z)$ is atomic.
\end{lemma}
\begin{proof}
Uniqueness of $z$ is straightforward to check.

Concerning the existence of $z$, if $M$ is diffuse, we set $z=1$. Thus, let us assume that $M$ has atoms. We define then $\mathcal Z$ which is the set of families of atoms $(e_i)_{i\in I}\subset M$ such that $z(e_i)z(e_j)=$ for all $i\not=j$. 
The set $\mathcal Z$ is non-empty since $M$ has atoms, and is ordered by inclusion: $(e_i)_{i\in I}\leq (f_j)_{j\in J}$ if and only if $I\subset J$ and $e_i=f_i$ for every $i\in I$. By Zorn's Lemma, we choose a maximal element $(e_i)_{i\in I}\in\mathcal Z$ and set
\[
1-z=\sum_{i\in I}z(e_i).
\]
Then one verifies easily that $Mz$ is diffuse and that $M(1-z)$ is atomic, by maximality.
\end{proof}

We need the following definition, which is due to J. von Neumann \cite[p. 482]{JvN}; see also \cite{vNWi}.

\begin{definition}
A topological group $G$ is \textit{minimally almost periodic} if its only continuous, finite-dimensional, irreducible, unitary representation is the one-dimensional trivial representation.
\end{definition}

It follows from \cite[Theorem 16.2.1]{DiC} that $G$ is minimally almost periodic if and only if $\AP(G)=\C$. For the reader's convenience, we recall some parts of the latter theorem in Theorem \ref{apgroup}.

The following result, which is used below for finite von Neumann algebras, seems to be new, as far as we know. We are grateful to P. de la Harpe for his suggestions in the treatment of the general case. In order to state it, let us recall from \cite[Chapter V]{Tak1} that an arbitrary von Neumann algebra $M$ acting on the separable, infinite-dimensional Hilbert space $\Hh$ admits the following direct sum decomposition:
\[
M=M_{\operatorname{I_f}}\oplus M_{\operatorname{I_\infty}}\oplus M_{\operatorname{II}}\oplus M_{\operatorname{III}}
\]
where
\begin{enumerate}
\item [(1)] $M_{\operatorname{I_f}}$ is a direct sum 
\[
M_{\operatorname{I_f}}=\bigoplus_{j\geq 1}A_j\otimes M_j(\C)
\]
with $A_j$ abelian and, for every $j\geq 1$, where $M_j(\C)$ is the finite factor of type $\operatorname{I}_j$; 
\item [(2)] $M_{\operatorname{I_\infty}}=A\overline{\otimes} B(\K)$ with $A$ abelian and where $\K$ is the separable, infinite-dimensional Hilbert space and where $N\overline{\otimes}P$ denotes the usual von Neumann algebra tensor product of the von Neumann algebras $N$ and $P$;
\item [(3)] $M_{\operatorname{II}}$ and $M_{\operatorname{III}}$ are the type $\operatorname{II}$ and $\operatorname{III}$ components of $M$ respectively. They are both diffuse.
\end{enumerate}

Lemma \ref{diffuse} implies that each $A_j$ in (1) is a direct sum $A_j=C_j\oplus D_j$ where $C_j$ is atomic, hence isomorphic either to $\ell^\infty(\N)$ or to $\C^{m_j}$ for some positive integer $m_j$, and where $D_j$ is diffuse. Rearranging the components of $M_{\operatorname{I_f}}$, we see that it is expressed as follows:
\begin{equation}\label{Ifin}
M_{\operatorname{I_f}}=\bigoplus_{k\geq 0}M_k
\end{equation}
where $M_0$ is a direct sum of tensor products $M_0=\bigoplus_{\ell\geq 1} D_\ell\otimes M_{p_\ell}(\C)$ with $D_\ell$ abelian and diffuse for every $\ell$, and where $M_k=M_{n_k}(\C)$ with $1\leq n_k<\infty$ for every $k\geq 1$.

\begin{theorem}\label{map}
Let $M\subset B(\Hh)$ be a von Neumann algebra acting on the separable, infinite-dimensional Hilbert space $\Hh$. Let 
\[
M=M_{\operatorname{I_f}}\oplus M_{\operatorname{I_\infty}}\oplus M_{\operatorname{II}}\oplus M_{\operatorname{III}}
\]
be the above decomposition of $M$, with $M_{\operatorname{I_f}}=\bigoplus_{k\geq 0}M_k$ as in (\ref{Ifin}).
Then the unitary group $U_M$ is minimally almost periodic, in other words $\AP(U_M)=\C$, if and only if the atomic part $\bigoplus_{k\geq 1}M_k$ of $M_{\operatorname{I_f}}$ is equal to $0$.
\end{theorem}
\begin{proof}
If $1=\sum_\ell z_\ell$ is a partition of the unity with $z_\ell\in \Pp(Z(M))$ for every $\ell$, then $U_M$ decomposes as 
\[
U_M=\prod_\ell U_{Mz_\ell},
\]
and $U_M$ is minimally almost periodic if and only if each $U_{Mz_\ell}$ is. Hence, if the atomic part $\bigoplus_{k\geq 1}M_k$ of $M_{\operatorname{I_f}}$ is non-trivial, $U_M$ is not minimally almost periodic.

Thus the proof will be complete if we prove that all groups $U_{M_X}$ for $X \in\{0,\operatorname{I_\infty},\operatorname{II},\operatorname{III} \}$ are minimally almost periodic.

We divide the proof into three parts. 
\\
(i) Assume that $A$ is an abelian and diffuse von Neumann algebra, so that it is $*$-isomorphic to $L^\infty [0,1]$. It suffices to prove that the only continuous character $\chi:U_A\rightarrow \T$ is the trivial one. By spectral theory, the subset 
\[
S:=\Big\{\exp\Big(i\frac{2k\pi}{m}p\Big) \colon m\geq 1, 0\leq k\leq m-1, p\in \Pp(A)\Big\}
\]
generates a dense subgroup of $U_A$ in the norm topology. Thus, let us fix a continuous character $\chi$ of $U_A$, and let $u=\exp(i2k\pi p/m)\in S$. Choose $\ep>0$ small enough so that $1$ is the only $m$-th root $\om$ of the unity such that $|\om-1|<\ep$, and next choose $\delta>0$ small enough so that one has $|\chi(v)-1|<\ep$ for every $v\in U_A$ such that $\Vert v-1\Vert_2<\delta$.

Since $A$ is diffuse, there exist pairwise orthogonal projections $q_1,\ldots,q_n\in \Pp(A)$ such that $p=\sum_j q_j$ and that 
\[
\Big\Vert \exp\Big(i\frac{2k\pi}{m}q_j\Big) -1\Big\Vert_2<\delta
\]
for every $j$. Then $\chi\Big(\exp\Big(i\frac{2k\pi}{m}q_j\Big)\Big)$ is an $m$-th root of the unity which is at distance within $\ep$ to $1$, hence it is equal to $1$, and finally 
\[
\chi\Big(\exp\Big(i\frac{2k\pi}{m}p\Big)\Big)=\prod_{j=1}^n \chi\Big(\exp\Big(i\frac{2k\pi}{m}q_j\Big)\Big)=1.
\]
\\
(ii) If $N$ is a diffuse von Neumann algebra, every unitary $u\in U_N$ belongs to some maximal abelian $*$-subalgebra $A=A(u)$ of $N$, which is necessarily diffuse by maximality. Hence, if $\pi$ is a continuous, finite-dimensional, irreducible, unitary representation of $U_N$, and if $u\in U_N$, the restriction of $\pi$ to $U_{A(u)}$ is a direct sum of irreducible ones, hence it is trivial by Part (i). In particular, $\pi(u)=1$. As this is true for every $u$, the representation $\pi$ is trivial. As the components $M_{\operatorname{II}}$ and $M_{\operatorname{III}}$ of $M$ are diffuse, this proves that the groups $U_{M_X}$ are minimally almost periodic for $X\in \{0,\operatorname{II},\operatorname{III}\}$. This also proves that $U_N$ is minimally almost periodic if $N=D\overline{\otimes}B(\K)$ with $D$ abelian and diffuse.
\\
(iii) Assume at last that $N$ is a von Neumann algebra of type $\operatorname{I_\infty}$ with atomic centre. By \cite[Theorem V.1.27]{Tak1}, $N=A\overline{\otimes} B(\K)$ with $A$ abelian and atomic. It suffices to prove that the unitary group $U(\K)$ is minimally almost periodic, which seems to be known by the experts. We sketch a proof for the sake of completeness. Let us first recall the description of all strongly continuous, irreducible unitary representations of $U(\K)$ (see for instance \cite[Proposition 9]{BekkaT}): Let $\rho:U(\K)\rightarrow U(\K)$ be the tautological representation of $U(\K)$. Then every strongly continuous, irreducible unitary representation $\pi$ of $U(\K)$ is unitarily equivalent to a subrepresentation of the representation $\rho^{\otimes k}\otimes \bar{\rho}^{\otimes\ell}$ of $U(\K)$ on $\K^{\otimes k}\otimes \bar{\K}^{\otimes\ell}$ for some integers $k,\ell\geq 0$, where the case $k+\ell=0$ corresponds to the trivial representation.

If $(\pi,\Ll)$ is a non-trivial strongly continuous, irreducible unitary representation of $U(\K)$, then, conjugating $\pi$ by a suitable unitary, we assume that $\Ll$ contains a unit vector of the form $\xi^{\otimes k}\otimes\bar{\xi}^{\otimes\ell}$ with $k+\ell>0$. Then the orbit $\pi(U(\K))\xi^{\otimes k}\otimes\bar{\xi}^{\otimes\ell}=\{(u\xi)^{\otimes k}\otimes\overline{(u\xi)}^{\otimes\ell} \colon u\in U(\K)\}$ contains an infinite orthonormal system, and $\pi$ is infinite-dimensional.
\end{proof}

We are ready to determine the von Neumann algebras $M\subset B(\Hh)$ such that $M\subset \ap_\Hh(M)$.

\begin{theorem}\label{ap}
Let $M\subset B(\Hh)$ be a von Neumann algebra. Then $M$ is contained in $\ap_\Hh(M)$ if and only if $M$ is isomorphic to a direct sum $\bigoplus_{k\geq 0} M_k$ where $M_0$ is an abelian, diffuse von Neumann algebra, and where $M_k$ is a finite-dimensional factor for every $k\geq 1$.
\end{theorem}
\begin{proof}
Assume first that $M$ is contained in $\ap_\Hh(M)$. As the latter space is contained in $\wap_\Hh(M)$, $M$ is a finite von Neumann algebra by Theorem \ref{3.1}. Let $z$ be the central projection such that $Mz$ is diffuse and $M(1-z)$ is atomic. We are going to prove that $Mz=Z(M)z$. In order to do that, let $x=xz\in Mz$. It suffices to prove that $\xi\star x\star\eta$ is constant for all $\xi,\eta\in\Hh$. Indeed, if it is the case, then $x\in M'\cap Mz=Z(M)z$ which is abelian and diffuse. Thus, let us fix $\xi,\eta\in \Hh$. From the equality $\xi\star x\star\eta=\xi\star xz\star\eta=(z\xi)\star x\star (z\eta)$, we assume further that $\xi=z\xi$ and $\eta=z\eta$. Replacing $M$ by $Mz$, we assume henceforth that $M$ is diffuse.

As $x\in \ap_\Hh(M)$, the orbit $U_{M}\xi\star x\star\eta$ is relatively compact. By Theorems \ref{map} and \ref{apgroup}, $\xi\star x\star\eta$ is constant.

Conversely, if $M$ is isomorphic to a direct sum $M_0\oplus\bigoplus_{k\geq 1} M_k$ where $M_0$ is an abelian, diffuse von Neumann algebra and each $M_k$ is a factor of type $\mathrm{I}_{n_k}$, say, with $1\leq n_k<\infty$ for every $k$, let $(z_k)_{k\geq 0}$ be the partition of $1$ formed by the central projections such that $Mz_k$ is isomorphic to $M_k$ for all $k\geq 0$. Thus, $\Hh=\bigoplus_{k\geq 0}\Hh_k$ with $\Hh_k=z_k\Hh$ for all $k$. Let then $x=\bigoplus_{k\geq 0}x_k\in M$ and $\xi=\bigoplus_k \xi_k,\eta=\bigoplus_k \eta_k\in \Hh$. Then with respect to the norm topology in $\cbr(U_M)$,
\[
\xi\star x\star\eta=\langle x_0\xi_0|\eta_0\rangle+\lim_{N\to\infty}\xi^{(N)}\star x^{(N)}\star\eta^{(N)}
\]
with $\xi^{(N)}\star x^{(N)}\star\eta^{(N)}\in \ap(U_{M^{(N)}})$ where $M^{(N)}=\bigoplus_{k=1}^N M_k$ is finite-dimensional. Indeed, one has for every $u=\bigoplus_k u_k\in U_M$:
\[
\xi\star x\star\eta(u)=\sum_{k=0}^N \langle u_kxz_ku_k^*\xi_k|\eta_k\rangle
+\sum_{k>N}\langle u_k xz_ku_k^*\xi_k|\eta_k\rangle
\]
and 
\[
\Big|\sum_{k>N}\langle u_kx_ku_k^*\xi_k|\eta_k\rangle\Big|\leq\Vert x\Vert \sum_{k>N}\Vert\xi_k\Vert\Vert\eta_k\Vert\to_{N\to\infty}0
\]
by Cauchy-Schwarz inequality since $\sum_k \Vert\xi_k\Vert^2<\infty$ and $\sum_k \Vert \eta_k\Vert^2<\infty$.
\end{proof}

Contrary to $\wap_\Hh(M)$ which always contains at least $K(\Hh)+\C$, $\ap_\Hh(M)$ can be very small, as the following proposition shows.

\begin{proposition}\label{trivial}
Suppose that $\Hh$ is infinite-dimensional. Then 
\[
\ap_\Hh(\B(\Hh))=\C.
\]
\end{proposition}
\begin{proof}
We already know from Theorem \ref{2.6} that $\ap_\Hh(B(\Hh))\subset K(\Hh)+\C$. Thus, let $T\in K(\Hh)$ be a non-zero, positive operator. Let $(\lm_j)_{j\geq 1}\subset \R^*_+$ be the sequence of positive eigenvalues of $T$ so that $\lm_1\geq \lm_2\geq\ldots$ and $\lm_j\to_{j\to\infty}0$. Let $(\ep_j)_{j\geq 1}\subset \Hh$ be an orthonormal system such that $T\ep_j=\lm_j\ep_j$ for every $j$. Then, denoting by $P_0$ the orthogonal projection onto $\ker(T)$ and by $P_\xi$ the rank-one projection onto $\C\xi$ for every unit vector $\xi$, we have
\[
T=\sum_{j\geq 1}^{\infty}\lm_j P_{\ep_j}\quad \textrm{and}\quad TP_0=P_0T=0
\]
where the series converges in the norm topology. Thus we get for all $u,v\in U(\Hh)$:
\begin{align*}
v\cdot \ep_1\star T\star\ep_1(u)
&=
(v\ep_1)\star T\star (v\ep_1)(u)=\langle Tu^*v\ep_1|u^*v\ep_1\rangle\\
&=
\sum_{j\geq 1} \lm_j \langle P_{\ep_j}(u^*v\ep_1)|u^*v\ep_1\rangle\\
&=
\sum_{j\geq 1} \lm_j|\langle u^*v\ep_1|\ep_j\rangle|^2
=
\sum_{j\geq 1} \lm_j|\langle v\ep_1|u\ep_j\rangle|^2.
\end{align*}
Let then $N>1$ be large enough so that $0\leq \lm_j\leq \lm_1/2$ for every $j\geq N$, and define $(v_n)_{n\geq N}\subset U(\Hh)$ such that $v_n$ is the identity on $P_0\Hh$ and
\[
v_n\ep_j=
\begin{cases}
\ep_j & j\not=1,n\\
\ep_n & j=1\\
\ep_1 & j=n.
\end{cases}
\]
Then, as $\ep_1\star T\star\ep_1(1)=\lm_1$, and as $v_n^*v_m\ep_1=v_n^*\ep_m=\ep_m$
for all $n,m\geq N$, $n\not=m$, we get:
\begin{align*}
\Vert v_m\cdot \ep_1\star T\star\ep_1-v_n\cdot \ep_1\star T\star\ep_1\Vert_\infty
&=
\Vert \ep_1\star T\star\ep_1-v_m^*v_n\cdot \ep_1\star T\star\ep_1\Vert_\infty\\
&\geq
|\ep_1\star T\star\ep_1(1)-\ep_1\star T\star\ep_1(v_n^*v_m)|\\
&=
\Big|\lm_1-\sum_j\lm_j\langle \ep_m|\ep_j\rangle\Big|\\
&=
\lm_1-\lm_m\geq \lm_1/2>0.
\end{align*}
This shows that the orbit $U(\Hh)\ep_1\star T\star\ep_1$ is not relatively compact in the norm topology of $C_b(U(\Hh))$, hence that $T\notin \ap_\Hh(B(\Hh))$.
\end{proof}

\begin{remark}
As for $\wap_\Hh(M)$, we do not know whether $\ap_\Hh(M)$ is a $C^*$-algebra.
\end{remark}

\section{Appendix: Weakly almost periodic functions on topological groups}

As promised in \S 1, the aim of this appendix is to give a sketched proof of the existence of a unique invariant mean on $\WAP(G)$, where $G$ is an arbitrary topological group. We keep notation that were settled in \S 1.

Let us recall first the following two theorems of A. Grothendieck.

\begin{theorem}\label{convcompacte}\cite[Th\'eor\`eme 5]{Grothen}
Let $\Om$ be a compact space and let $A\subset C(\Om)$ be a bounded set. Then $A$ is relatively weakly compact if and only if, for every sequence $(f_n)\subset A$, there exists a subsequence $(f_{n_k})$ and an element $h\in C(\Om)$ such that 
\[
\lim_k f_{n_k}(\om)=h(\om)
\]
for every $\om\in\Om$.
\end{theorem}
\begin{proof} (Sketch) The proof rests on Eberlein-Smulian theorem \cite[Theorem A.12]{Glas} which states that if $A$ is a bounded subset of a Banach space $X$, then $A$ is relatively weakly compact if and only if every sequence in $A$ has a subsequence which converges weakly in $X$.

Thus, if a bounded sequence $(f_n)$ converges pointwise to the limit $h$, then, by Lebesgue theorem, $\int f_nd\mu\to \int hd\mu$ for every regular complex measure $\mu$ on $\Om$. Compactness of $\Om$ implies that every continuous linear functional on $C(\Om)$ is such a measure, hence relative compactness in the pointwise convergence topology implies relatively weak compactness. The converse is obvious, as every linear form of the type $f\mapsto f(\om)$ is weakly continuous.
\end{proof}

Using the Stone-Cech compactification of $G$, A. Grothendieck gets the following theorem which we state in a slightly different, though equivalent, form.
 
\begin{theorem}\label{Grothendieck}\cite[Th\'eor\`eme 6]{Grothen}
Let $G$ be an arbitrary topological group. Then a bounded subset $A$ of $C_b(G)$ is weakly relatively compact if and only if for all sequences 
$(x_i)\subset G$ and $(f_j)\subset A$ such that the following limits 
\[
\lim_i(\lim_j f_j(x_i))\quad\textrm{and}\quad
\lim_j(\lim_i f_j(x_i))
\]
both exist, then they are equal.
\end{theorem}

As a consequence, he gets the following criterion for weakly almost periodic functions on topological groups.

\begin{proposition}\label{criterewap}
Let $G$ be a topological group and let $f\in C_{b,r}(G)$. Then $f$ is weakly almost periodic if and only if for all sequences $(x_i),(y_j)$ in $G$ such that the following limits 
\[
\lim_i(\lim_j f(x_iy_j))\quad\textrm{and}\quad
\lim_j(\lim_i f(x_iy_j)),
\]
both exist, then they are equal. In particular, the left orbit $Gf$ is relatively weakly compact if and only if the right orbit $fG$ is.
\end{proposition}

\begin{proposition}\label{4.3}
Let $G$ be as above. The set $\WAP(G)$ is a unital $C^*$-subalgebra of $\cbr(G)$ which is left- and right-invariant under translations of $G$. 
\end{proposition}
\begin{proof}
It is clear that $\WAP(G)$ contains all constant functions, and that $\overline{f}\in\WAP(G)$ if $f\in\WAP(G)$. Moreover, for fixed $g\in G$, the maps $f\mapsto g\cdot f$ and $f\mapsto f\cdot g$ are clearly weakly continuous, hence $g\cdot f,f\cdot g\in \WAP(G)$ if $f\in\WAP(G)$.

In order to prove that $\WAP(G)$ is a $C^*$-algebra, we are going to apply Theorem \ref{convcompacte}.
In order to do that, let $\Om$ be the Gelfand spectrum of the unital $C^*$-algebra $\cbr(G)$. It is a compact space, and Gelfand transform $f\mapsto \hat{f}: \chi\mapsto \chi(f)$ is a $*$-isomorphism from $\cbr(G)$ onto $C(\Om)$, and the weak topology on $\cbr(G)$ corresponds to that on $C(\Om)$. Furthermore, $G$ acts continuously on $\Om$ by $g\cdot\chi(f)=\chi(g^{-1}\cdot f)$. Indeed, if $\chi_i\to \chi$ and $g_j\to 1$, one has
\begin{align*}
|\chi(f)-\chi_i(g_j^{-1}\cdot f)|
&\leq
|\chi(f)-\chi_i(f)|+|\chi_i(f)-\chi_i(g_j^{-1}\cdot f)|\\
&\leq
|\chi(f)-\chi_i(f)|+\Vert f-g_j^{-1}\cdot f\Vert_\infty\\
=&
|\chi(f)-\chi_i(f)|+\Vert g_j\cdot f-f\Vert_\infty\to 0
\end{align*}
as $i,j\to\infty$. Thus the image $\widehat{\WAP(G)}$ of $\WAP(G)$ under the Gelfand transform is exactly the set of elements $f\in C(\Om)$ for which $Gf$ is relatively weakly compact. The fact that $\WAP(G)$ is a $*$-algebra is a straightforward consequence of Theorem \ref{convcompacte}.

Let us prove finally that $\WAP(G)$ closed in $\cbr(G)$ or, what amounts to be the same, that $\widehat{\WAP(G)}$ is closed in $C(\Om)$: let $(f_n)_{n\geq 1}\subset \widehat{\WAP(G)}$ be a sequence which converges to $f\in C(\Om)$. Let us show that $Gf$ is relatively weakly compact. In order to do that, let $(g_k)_{k\geq 1}\subset G$ be a sequence. We are going to prove that there exists a subsequence $(g_{k_j})_{j\geq 1}\subset (g_k)$ and an element $h\in C(\Om)$ such that $g_{k_j}\cdot f(\om)\to h(\om)$ for every $\om\in\Om$. By the standard diagonal process, there exist a subsequence $(g_{k_j})\subset (g_k)$ and a sequence $(h_\ell)_{\ell\geq 1}\subset C(\Om)$ such that
\[
\lim_{j\to\infty} g_{k_j}\cdot f_\ell(\om)=h_\ell(\om)
\]
for every $\om\in\Om$ and every $\ell\geq 1$. It is easy to prove that $(h_\ell)$ is a Cauchy sequence in $C(\Om)$ and then that $(h_\ell)$ converges in norm to a limit denoted by $h\in C(\Om)$. Finally, one proves that, for every $\om\in\Om$, $g_{k_j}\cdot f(\om)\to_{j\to\infty}h(\om)$.
\end{proof}

Here is now the promised theorem on the existence of the unique bi-invariant mean on $\WAP(G)$.

\begin{theorem}\label{moy}
There exists a unique linear functional $\mf:\WAP(G)\rightarrow\C$ with the following properties:
\begin{enumerate}
\item [(a)] $\mf(f)\geq 0$ for every $f\geq 0$;
\item [(b)] $\mf(1)=1$;
\item [(c)] $\mf(g\cdot f)=\mf(f\cdot g)=\mf(f)$ for all $f\in\WAP(G)$ and 
$g\in G$;
\item [(d)] for every $f\in\WAP(G)$ and every $\ep>0$, there exists a convex combination $\psi:=\sum_{j=1}^m t_j{g_j}\cdot f$ (with $g_j\in G$ and $t_j\geq 0$, $\sum_j t_j=1$) such that 
$
\Vert\psi-\mf(f)\Vert_\infty<\ep,
$
and there exists a convex combination $\f:=\sum_i s_i f\cdot h_i$ (with $h_i\in G$ and $s_i\geq 0$, $\sum_i s_i=1$) such that
$
\Vert \f-\mf(f)\Vert_\infty<\ep.
$
\end{enumerate}
\end{theorem}
\begin{proof} For $f\in\WAP(G)$, let us denote by $Q_l(f)=\overline{\co}(Gf)$ the norm closed convex hull of $Gf$ and similarly $Q_r(f)=\overline{\co}(fG)$.
The group $G$ acts by left translations on $Q_l(f)$ which are affine transformations and are weakly continuous since, for every fixed $g$, the map $f\mapsto g\cdot f$ is linear and isometric.
Moreover, if $\psi_1,\psi_2\in Q_l(f)$ are such that $\psi_1\not=\psi_2$, then
$0\notin \overline{\{ g\cdot\psi_1-g\cdot \psi_2 \colon g\in G\}}^{\Vert\cdot\Vert_\infty}$ and thus the action of $G$ on $Q_l(f)$ is distal. By Ryll-Nardzewski Theorem, there exists $c_l(f)\in Q_l(f)$ such that $g\cdot c_l(f)=c_l(f)$ for every $g\in G$. This means that $c_l(f)$ is constant.

Similarly, $G$ acts on the right on $Q_r(f)$, and by the same arguments, $Q_r(f)$ contains a constant $c_r(f)$. Using convex approximations of both constants, it is easy to check that $Q_l(f)$ and $Q_r(f)$ contain the same constant which is then unique. 

Thus, by definition, $\mf(f)=c_l(f)=c_r(f)$ for every $f\in \WAP(G)$.

If a linear form $\mf':\WAP(G)\rightarrow\C$ satisfies (a) and (b) and if $\mf'(g\cdot f)=\mf'(f)$ for all $f\in\WAP(G)$ and $g\in G$, then $\mf'$ is constant on $Q_l(f)$ and we infer that $\mf'(f)=c_l(f)$. If $\mf''$ is another left-invariant mean, one has necessarily $\mf''(f)=\mf'(f)$ by the above remarks. This proves uniqueness of $\mf$.
Furthermore, the fact that $\mf(f)$ belongs to $Q_l(f)\cap Q_r(f)$ implies (d) and then the right invariance of $\mf$. 

Properties (a), (b), (c) and (d) are obvious, as well as the fact that $\mf(\alpha f)=\alpha\mf(f)$ for all $\alpha\in \C$ and $f\in \WAP(G)$. We are left to prove that, for $f_1,f_2\in\WAP(G)$, one has $\mf(f_1+f_2)=\mf(f_1)+\mf(f_2)$. Let us fix $\ep>0$. There
exist $s_1,\ldots,s_m>0$ such that $\sum_i s_i=1$, and $g_1,\ldots,g_m\in G$ such that 
\[
\Big\Vert\sum_i s_ig_i\cdot f_1-\mf(f_1)\Big\Vert_\infty\leq\frac{\varepsilon}{2}.
\]
for every $g\in G$. But
$
\mf(f_2)=\mf\Big(\sum_i s_i g_i\cdot f_2\Big).
$
Indeed, $\sum_i s_i g_i\cdot f_2\in Q_l(f_2)$, and as the latter set is convex, we have
\[
Q_l\Big(\sum_i s_i g_i\cdot f_2\Big)\subset Q_l(f_2).
\] 
Hence the constant in the left-hand convex set is equal to the one in $Q_l(f_2)$. Thus there exist $t_1,\ldots, t_n>0$ such that $\sum_j t_j=1$, and $h_1,\ldots,h_n\in G$ such that 
\begin{align*}
\lefteqn{\Big\Vert \mf(f_2)-\sum_{i,j}s_it_j (h_jg_i)\cdot f_2\Big\Vert_\infty}\\
&=\Big\Vert \mf\Big(\sum_i s_i g_i\cdot f_2\Big)-\sum_{j}t_jh_j\cdot\Big(\sum_i s_i g_i\cdot f_2\Big)\Big\Vert_\infty
\leq
\frac{\ep}{2}.
\end{align*}
As $\sum_{i,j}s_it_j=1$, one has
\[
\sum_{i,j}s_it_j(h_jg_i)\cdot[f_1+f_2]\in Q_l(f_1+f_2)
\]
and
\begin{align*}
\lefteqn{\Big\Vert \sum_{i,j}s_it_j(h_jg_i)\cdot[f_1+f_2]-\mf(f_1)-\mf(f_2)\Big\Vert_\infty
}\\
&\leq
\Big\Vert \sum_{i,j}s_it_j(h_jg_i)\cdot f_1-\mf(f_1)\Big\Vert_\infty
+
\Big\Vert \sum_{i,j}s_it_j(h_jg_i)\cdot f_2-\mf(f_2)\Big\Vert_\infty\\
&\leq
\sum_j t_j\Big\Vert\sum_i s_i h_j\cdot g_i\cdot f_1-\mf(f_1)\Big\Vert_\infty +
\frac{\ep}{2}\\
&=
\sum_j t_j\Big\Vert h_j\cdot \Big(\sum_i s_ig_i\cdot f_1-\mf(f_1)\Big)\Big\Vert_\infty
+\frac{\ep}{2}\leq 
\ep.
\end{align*}
This shows that $\mf(f_1)+\mf(f_2)\in Q_l(f_1+f_2)$.
The proof is now complete.
\end{proof}

We end the appendix with the following result on almost periodic functions on the topological group $G$ which is used in Theorems \ref{map} and \ref{ap}. It is part of \cite[Theorem 16.2.1]{DiC}.

\begin{theorem}\label{apgroup}
Let $G$ be a topological group and let $f\in C_b(G)$. Then the following conditions are equivalent:
\begin{enumerate}
\item [(1)] The left orbit $Gf$ is relatively compact in the norm topology, namely, $f$ is almost periodic.
\item [(2)] The right orbit $fG$ is relatively compact in the norm topology.
\item [(3)] The set $\{g\cdot f\cdot h\colon g,h\in G\}$ is relatively compact in the norm topology.
\item [(4)] The function $f$ is a uniform limit over $G$ of linear combinations of coefficients of finite-dimensional, irreducible continuous unitary representations of $G$.
\end{enumerate}
\end{theorem}

\par\vspace{2mm}

\bibliographystyle{plain}
\bibliography{refFWap}

\end{document}